\documentclass[12pt,letterpaper]{article}
\usepackage{a4, graphicx,amsmath,rotating,amssymb,epsfig,epstopdf,amsthm,etoolbox}
\usepackage[numbers]{natbib}
\usepackage{pslatex} 
\usepackage{graphicx}
\usepackage{amssymb}
\usepackage{epstopdf}
\usepackage{titlesec}
\usepackage{tikz,subfig}
\usetikzlibrary{shapes,positioning,intersections,quotes,svg.path,shapes.geometric}
\usepackage{algpseudocode,enumitem}
\usepackage{algorithm}
\usepackage[mathscr]{eucal}
\usepackage[margin=1in]{geometry}
\usepackage{mathtools}
\usepackage{hyperref}
\hypersetup{
    plainpages=false,       
    unicode=false,          
    pdftoolbar=true,        
    pdfmenubar=true,        
    pdffitwindow=false,     
    pdfstartview={FitH},    
    pdftitle={uWaterloo\ LaTeX\ Thesis\ Template},    
    pdfnewwindow=true,      
    colorlinks=true,        
    linkcolor=blue,         
    citecolor=blue,        
    filecolor=magenta,      
    urlcolor=cyan           
}
\makeatletter
\patchcmd{\NAT@test}{\else \NAT@nm}{\else \NAT@hyper@{\NAT@nm}}{}{}
\makeatother

\defcitealias{tutte72}{Tutte (1972) [cf. Bondy and Murty (1976)]}
\defcitealias{tutte54}{Tutte 1954}
\defcitealias{tutte66}{1966}

\theoremstyle{plain}
\newtheorem{theorem}{Theorem}[section]
\newtheorem{lemma}[theorem]{Lemma}

\newtheorem{conjecture}{Conjecture}
\newtheorem{definition}[theorem]{Definition}

\newtheorem{claim}[theorem]{Claim}
\newtheorem{prop}{}
\newenvironment{Prop}[2][]
{\renewcommand\theprop{#2}\begin{prop}[#1]}
{\end{prop}}

\newtheorem{inclaim}{Claim}
\newenvironment{Inclaim}[2][]
{\renewcommand\theinclaim{#2}\begin{inclaim}}
{\end{inclaim}}

\begin{document}

\title{Strong $3$-Flow Conjecture for Projective Planar Graphs}
\author{J. V. de Jong and R. B. Richter\footnote{Department of Combinatorics and Optimization, University of Waterloo, Canada (jamiev.dejong@gmail.com, brichter@uwaterloo.ca)}}

\maketitle
\begin{abstract}
In 1972, Tutte posed the $3$-Flow Conjecture: that all $4$-edge-connected graphs have a nowhere zero $3$-flow. This was extended by Jaeger et al.~(1992) to allow vertices to have a prescribed, possibly non-zero difference (modulo $3$) between the inflow and outflow. They conjectured that all $5$-edge-connected graphs with a valid prescription function have a nowhere zero $3$-flow meeting that prescription. Kochol (2001) showed that replacing $4$-edge-connected with $5$-edge-connected would suffice to prove the $3$-Flow Conjecture and Lov\'asz et al.~(2013) showed that both conjectures hold if the edge connectivity condition is relaxed to $6$-edge-connected. Both problems are still open for $5$-edge-connected graphs. \\

The $3$-Flow Conjecture was known to hold for planar graphs, as it is the dual of Gr\"otzsch's Colouring Theorem. Steinberg and Younger (1989) provided the first direct proof using flows for planar graphs, as well as a proof for projective planar graphs. Richter et al.~(2016) provided the first direct proof using flows of the Strong $3$-Flow Conjecture for planar graphs. We prove the Strong $3$-Flow Conjecture for projective planar graphs. 
\end{abstract}
\section{Introduction}

A \emph{$\mathbb{Z}_k$-flow} on a graph $G$ is a function that assigns to each edge $e\in E(G)$ an ordered pair consisting of a direction, and a value $f(e)\in\{0,...,k-1\}$, such that if $D$ is the resulting directed graph, then, for each vertex $v\in V(G)$,
$$\sum_{e=(u,v)\in E(D)}f(e)-\sum_{e=(v,w)\in E(D)}f(e)\equiv 0\pmod{k}.$$ 
It is easy to see that every graph $G$ has a $\mathbb{Z}_k$-flow for every value of $k$: set $f(e)=0$ for all $e\in E(G)$. Therefore it is typical to use the following more restrictive concept. A \emph{nowhere zero $\mathbb{Z}_k$-flow} on $G$ is a $\mathbb{Z}_k$-flow on $G$ such that no edge is assigned the value zero. In 1950, Tutte proved that a graph has a nowhere-zero $\mathbb{Z}_k$-flow if and only if it has a nowhere-zero $k$-flow, which requires the net flow through each vertex to be exactly zero; see \citet{diestal05} for a proof of this equivalence and further background on flows.\\

Tutte (cf.~\citet{tutte72}) conjectured that every $4$-edge-connected graph has a nowhere zero $3$-flow. This is known as the $3$-Flow Conjecture, and while progress has been made for many classes of graphs, it is still an open problem. For planar graphs the $3$-Flow Conjecture is equivalent to Gr\"otzsch's Theorem, and \citet{steinbergyounger89} provided a direct proof using flows. \citet{steinbergyounger89} also proved that the $3$-Flow Conjecture holds for graphs embedded in the projective plane. \\

As an extension of $\mathbb{Z}_3$-flows, we add a prescription function, where each vertex in the graph is assigned a value in $\mathbb{Z}_3$ that defines the net flow through the vertex. The prescriptions of the vertices in the graph must sum to zero in $\mathbb{Z}_3$. A graph $G$ is $\mathbb{Z}_3$-connected if for each valid prescription function $p$, $G$ has a nowhere-zero $\mathbb{Z}_3$-flow achieving $p$. This led \citet{jaegar92} to pose the following conjecture.

\begingroup
\def\thetheorem{\ref{jaeger}}
\begin{conjecture}[Strong $3$-Flow Conjecture]
Every $5$-edge-connected graph is $\mathbb{Z}_3$-connected.
\end{conjecture}
\endgroup

\citet{jaegar88} had earlier posed the following weaker conjecture.

\begin{conjecture}[Weak 3-Flow Conjecture]
There is a natural number $h$ such that every $h$-edge-connected graph has a nowhere zero $\mathbb{Z}_3$-flow. 
\end{conjecture}

\citet{jaegar88} also showed that this conjecture is equivalent to the same statement regarding $\mathbb{Z}_3$-connectivity. The Weak 3-Flow Conjecture remained open until \citet{thomassen12} proved that $h=8$ sufficed. 

\begin{theorem}
Every $8$-edge-connected graph is $\mathbb{Z}_3$-connected. 
\end{theorem}

\citet{lovasz13} extended this to the following result. 

\begin{theorem}
If $G$ is a $6$-edge-connected graph, then $G$ is $\mathbb{Z}_3$-connected.
\end{theorem}

\citet{kochol01} showed that the $3$-Flow Conjecture is equivalent to the statement that every $5$-edge-connected graph has a nowhere zero $\mathbb{Z}_3$-flow, so the result of \citet{lovasz13} is one step away from the $3$-Flow Conjecture. As stated, both of these results also make significant steps toward the Strong $3$-Flow Conjecture, as they considered $\mathbb{Z}_3$-connectivity. \citet{laili06} proved the Strong $3$-Flow Conjecture holds for planar graphs using the duality with graph colouring. \citet{richteretal16} provided the first direct proof of this result using flows. Their result is Theorem \ref{mainplanar}. 

\begin{theorem}
\label{mainplanar}
Let $G$ be a $3$-edge-connected graph embedded in the plane with at most two specified vertices $d$ and $t$ such that
\begin{itemize}
\item if $d$ exists, then it has degree $3$, $4$, or $5$, has its incident edges oriented and labelled with elements in $\mathbb{Z}_3\setminus\{0\}$, and is in the boundary of the unbounded face,
\item if $t$ exists, then it has degree $3$ and is in the boundary of the unbounded face,
\item there are at most two $3$-cuts, which can only be $\delta(\{d\})$ and $\delta(\{t\})$,
\item if $d$ has degree $5$, then $t$ does not exist, and
\item every vertex not in the boundary of the unbounded face has five edge-disjoint paths to the boundary of the unbounded face. 
\end{itemize}
If $G$ has a valid prescription function, then $G$ has a valid $\mathbb{Z}_3$-flow. 
\end{theorem}

The $3$-Flow Conjecture and the Strong $3$-Flow Conjecture for planar graphs are corollaries of Theorem~\ref{mainplanar}.\\

In this paper we prove the Strong $3$-Flow Conjecture for projective planar graphs, extending Theorem \ref{mainplanar}. This proof requires the preliminary results that appear in \citet{paper1}. All of these results also appear in the first author's Ph.D. thesis \citep{thesis}. The techniques used for the proof in this paper build off those of \citet{thomassen12}, \citet{lovasz13}, \citet{richteretal16}, and \citet{steinbergyounger89}.  \\

In Section \ref{projprelim} we first discuss some of the ideas that will be used throughout this paper. 

\section{Preliminaries}
\label{projprelim}

\subsection*{Basic Definitions}
We define an \emph{orientation} of a graph $G$ to be the directed graph $D$ obtained by adding a direction to each edge. With reference to determining the $\mathbb{Z}_3$-connectivity of $G$ with prescription function $p:V\rightarrow\{-1,0,1\}$, an orientation is \emph{valid} if for each vertex $v\in V(G)$,
$$\sum_{e=(u,v)\in E(D)}1-\sum_{e=(v,w)\in E(D)}1\equiv p(v)\pmod{3}.$$ 
It can be easily seen that this is equivalent to a $\mathbb{Z}_3$-flow.\\

A vertex $d$ in a graph $G$ is a \emph{directed vertex} if all its incident edges are directed. We call this an \emph{orientation} of $d$. We say that an orientation of $G$ \emph{extends} the orientation of $d$ if the direction of the edges at $d$ is maintained. In cases involving a directed vertex we take the term \emph{valid orientation} to include that the orientation extends that of~$d$. \\

We \emph{lift} a pair of edges $uv$ and $vw$ in a graph $G$ by deleting $uv$ and $vw$, and adding an edge $vw$. We define an edge-cut $\delta(A)$ in $G$ to be \emph{internal}\index{internal edge-cut} if either $A$ or $G-A$ does not intersect the boundary of the specified face(s) of $G$. 

\subsection*{Specified Face(s)}

First, we consider the idea of a specified face. Theorem \ref{mainplanar} allows vertices of degree less than $5$ on the boundary of the outer face. However, a graph embedded in a higher genus surface does not have a defined outer face, so the result cannot be directly extended. \\

Let $G$ be a graph with a specified face $F_G$. Let $G'$ be a graph obtained from $G$ by one or more operations. Unless otherwise stated, the specified face $F_{G'}$ is defined as follows:
\begin{enumerate}
\item Suppose that $G'$ is obtained from $G$ by deleting or contracting a connected subgraph of $G$ that has no edge in common with the boundary of $F_G$. Then $F_{G'}=F_G$. 
\item Suppose that $G'$ is obtained from $G$ by contracting a connected subgraph of $G$ that contains the boundary of $F_G$. Then $G'$ has no specified face. A face can be chosen arbitrarily; in general we will chose a specified face incident with the vertex of contraction. 
\item Suppose that $G'$ is obtained from $G$ by deleting an edge $e$ in the boundary of~$F_G$. Let $F$ be the other face incident with~$e$. Then $F_{G'}$ is the face formed by the union of the boundaries of $F$ and $F_G$ (without $e$). 
\item Suppose that $G'$ is obtained from $G$ by deleting a vertex $v$ in the boundary of~$F_G$. Let $F_1$, $F_2$,...,$F_k$ be the other faces incident with~$v$. Then $F_{G'}$ is the face formed by the union of the boundaries of $F_1$, $F_2$,...,$F_k$, and $F_G$ (without the edges incident with $v$).
\item Suppose that $G'$ is obtained from $G$ by contracting a connected subgraph $H$ of $G$ whose intersection with $F_G$ is a path $P$ of length at least one. Then $F_{G'}$ is the face formed by the boundary of $F_G$ without $P$. If the intersection of $H$ with $F_G$ consists of more than one path, this contraction can simply be completed in multiple steps. 
\item Suppose that $G'$ is obtained from $G$ by lifting a pair of adjacent edges $e_1$, $e_2$, where $e_1$ is in the boundary of~$F_G$, $e_2$ is not, and $e_1$ and $e_2$ are consecutive at their common vertex. Let $F_1$ be the other face incident with~$e_1$. Note that $F_1$ is incident with~$e_2$. Let $F_2$ be the other face incident with~$e_2$. Then $F_{G'}$ is the face formed by the union of the boundaries of $F_G$ and $F_2$ (using the lifted edge instead of $e_1$ and $e_2$). 
\end{enumerate}
When performing these operations we will not explicitly state the new specified face unless necessary. 

\subsection*{Edge-Disjoint Paths to the Boundary}

As in \citet{richteretal16}, throughout this paper we are working with graphs for which all vertices not on the boundary of the specified face(s) have at least $5$ edge-disjoint paths to the boundary of the specified face(s). We describe here how reductions to the graph affect this condition. The proof of this result is straightforward, and can be found in \citep{thesis}.\\

\begin{lemma}
Let $G$ be a graph with specified face $F_G$ such that all vertices not on the boundary of $F_G$ have $5$ edge-disjoint paths to the boundary of~$F_G$. Let $G'$ be a graph obtained from $G$ by 
\begin{enumerate}
\item contracting a subgraph $X$ of $G$ that does not intersect the boundary of $F_G$ to a vertex~$x$,
\item deleting a boundary edge $e$ of~$F_G$,
\item deleting a boundary vertex $x$ of~$F_G$,
\item lifting a pair of adjacent edges $e_1$, $e_2$, where $e_1$ is in the boundary of $F_G$, $e_2$ is not, and $e_1$ and $e_2$ are consecutive at their common vertex, or
\item contracting a subgraph $X$ of $G$ whose intersection with $F_G$ is a path~$P$.
\end{enumerate}
Then all vertices not on the boundary of $F_{G'}$ have $5$ edge-disjoint paths to the boundary of~$F_{G'}$. 
\label{edgedisjoint}
\end{lemma}

We therefore only discuss the preservation of this property in cases where Lemma \ref{edgedisjoint} does not apply. 

\subsection*{Minimal Cuts}
Let $G$ be a graph with a directed vertex $d$. An edge-cut $\delta(A)$ in $G$ with $d\in A$ is \emph{$k$-robust} if $|A|\geq 2$ and $|G-A|\geq k$. \\

Throughout this paper we will perform local reductions on graphs. Many of these reductions will involve considering $2$-robust edge-cuts, either because $G$ has a small edge-cut that must be reduced, or because we must verify that the graph resulting from a reduction does not have any small edge-cuts. In all cases, we first consider the smallest possible edge-cuts. Thus, if we consider a $2$-robust $k$-edge-cut $\delta(A)$ in a graph $G$, we may assume that $G$ has no $2$-robust at most $(k-1)$-edge-cut. Thus it may be assumed that either $G[A]$ is connected, or it consists of two isolated vertices whose degrees sum to~$|\delta(A)|$. The same is true of~$G-A$. Given the sizes of the cuts we consider, generally both $G[A]$ and $G-A$ will be connected. 

\subsection*{Non-Crossing 3-Edge-Cuts}

Let $\delta(A)$ and $\delta(B)$ be distinct edge-cuts in $G$. We say that $\delta(A)$ and $\delta(B)$ \emph{cross} if $A\cap B$, $A\setminus B$, $B\setminus A$, and $\overline{A\cup B}$ are all non-empty. Throughout this paper we consider graphs that are allowed to have non-crossing $2$-robust $3$-edge-cuts under certain restrictions. The following well-known result allows us to assume that such cuts are non-crossing. A proof can be found in \citep{thesis}.

\begin{lemma}
Let $k$ be an odd positive interger. If $G$ is a $k$-edge-connected graph, then any two $k$-edge-cuts in $G$ do not cross. 
\end{lemma}

\subsection*{Face Boundaries in the Projective Plane}

Let $G$ be a graph embedded in the plane, and let $F_G$ be a specified face of~$G$. If $G$ does not contain a cut vertex, then we may assume that $F_G$ is bounded by a cycle. If $G$ is embedded in the projective plane, this is not necessarily true. Suppose that $v$ is a vertex that appears more than once in the boundary walk of $F_G$, and assume that $v$ is not a cut vertex. Then there exists a non-contractible curve that passes through only $F_G$ and~$v$. Cut along this curve, and draw the graph on the plane. The result is a planar graph with one specified face containing two copies of~$v$. See Figure \ref{added} for an illustration. Contract the two copies of $v$ to a single vertex. Then $G$ is a planar graph with two specified faces, each containing~$v$. The following result of \citet{paper1} guarantees that such a graph has a valid orientation. 

\begin{definition}
An \emph{FT graph} is a graph $G$ embedded in the plane, together with a valid prescription function $p: V(G)\rightarrow\{-1,0,1\}$, such that:
\begin{enumerate}
\item $G$ is $3$-edge-connected,
\item $G$ has two specified faces $F_G$ and $F_G^*$, and at most one specified vertex $d$ or~$t$,
\item there is at least one vertex in common between $F_G$ and~$F_G^*$,
\item if $d$ exists, then it has degree $3$, $4$, or $5$, is oriented, and is in the boundary of both $F_G$ and~$F_G^*$,
\item if $t$ exists, then it has degree $3$ and is in the boundary of at least one of $F_G$ and~$F_G^*$,
\item $G$ has at most one $3$-edge-cut, which can only be $\delta(\{d\})$ or $\delta(\{t\})$, and
\item every vertex not in the boundary of $F_G$ or $F_G^*$ has $5$ edge-disjoint paths to the union of the boundaries of $F_G$ and~$F_G^*$.
\end{enumerate}
\end{definition}

\begin{theorem}
Every FT graph has a valid orientation. 
\label{mainft}
\end{theorem}

\begin{figure}[b]
\begin{center}
\subfloat{
\begin{tikzpicture}
\draw[dashed] (0,0) circle (60pt);
\filldraw (1.05,0) circle (2pt) node[left] {$v$};
\filldraw (0,-1.05) circle (2pt);
\filldraw (0,1.05) circle (2pt);
\filldraw (0,0) circle (0pt) node {$F_G$};

\draw[thick] (0,1.05) to[out=0,in=90] (1.05,0);
\draw[thick, blue] (0,-1.05) to[out=0,in=-90] (1.05,0);



\draw (1.05,0) -- (2.04,0.5);
\draw (1.05,0) -- (2.04,-0.5);

\draw (-2.04,0.5) -- (0,1.05);
\draw (-2.04,-0.5) -- (0,-1.05);
\end{tikzpicture}
}
\subfloat{
\begin{tikzpicture}
\draw[thick] (0,0) to[out=0,in=90] (1.05,-1.05);
\draw[thick, blue] (2.10,0.5) to[out=180,in=-90] (1.05,1.55);
\draw (0,-2) circle (0pt);
\draw (-1.5,0) circle (0pt);
\filldraw (1.05,-1.05) circle (2pt) node[right] {$v$};
\filldraw (1.05,1.55) circle (2pt) node[right] {$v$};
\filldraw (0,0) circle (2pt);
\filldraw (2.10,0.5) circle (2pt);

\draw (0,0) -- (1.05,1.55);
\draw (2.10,0.5) -- (1.05,-1.05);
\end{tikzpicture}
}
\caption{Redrawing $G$ in the plane.}
\label{added}
\end{center}
\end{figure}
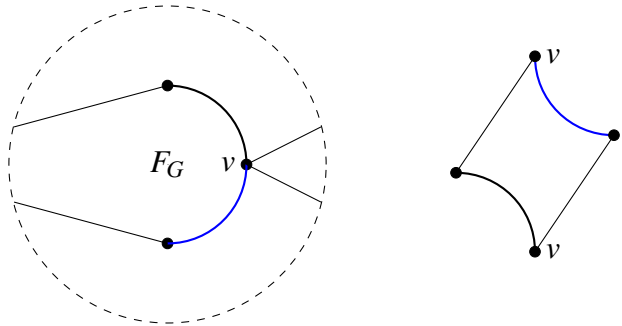

\subsection*{Chords in the Projective Plane}

Let $G$ be a graph embedded in the plane and let $F_G$ be a specified face of~$G$. Suppose that $F_G$ has a chord $e$ with endpoints $u$ and~$v$. Then there exist subgraphs $H$ and $K$ of $G$ such that $H\cap K=\{\{u,v\},\{e\}\}$ and $H\cup K=G$. This is a property that \citet{richteretal16} exploited when proving the Strong $3$-Flow Conjecture for planar graphs, and a property that we use throughout this paper. \\

Now suppose that $G$ is a graph embedded in the projective plane with a cycle~$C$ bounding a closed disk. Let $e$ be a chord of~$C$. If $C+e$ is contained in an open disk, then $e$ is a \emph{contractible} chord. Otherwise it is a \emph{non-contractible} chord. This is relevant in the case where the specified face $F_G$ is bounded by a cycle. If $F_G$ has a contractible chord, we may use techniques analogous to those in the plane to reduce the graph. In the case of a non-contractible chord, the graph is not split into subgraphs as it is in the plane, and we require different techniques. We see here that the deletion of such a chord and its endpoints results in a planar graph. 

\begin{lemma}
Let $G$ be a graph embedded in the projective plane with a specified face~$F_G$ bounded by a cycle. Suppose that $uv$ is a non-contractible chord of $F_G$. Let $G'=G-\{u,v\}$. Then $G'$ is planarly embedded in the projective plane with one specified face; namely the one containing $F_G$. 
\label{noncontr}
\end{lemma}
\begin{proof}
Consider the projective plane as represented by a circle of radius $2$, where opposite points are identified. Draw the given embedding of $G$ such that the boundary of $F_G$ lies on the unit circle (where the origin is in the specified face), and $e$ is contained in the line $y=0$. The graph $G'$ does not intersect the line $y=0$. Cut the projective plane along the line $y=0$ and identify the opposite points on the circle. The result is a planar embedding of $G'$ where $F_{G'}$ is the outer face. 
\end{proof}

In practice, we will obtain graphs with up to three vertices of degree three. The following result of \citet{paper1} shows that such graphs have a valid orientation. 

\begin{definition}A \emph{DTS graph} is a graph $G$ embedded in the plane, together with a valid $\mathbb{Z}_3$-prescription function $p: V(G)\rightarrow\{-1,0,1\}$, such that:
\begin{enumerate}
\item $G$ is $3$-edge-connected,
\item $G$ has a specified face $F_G$, and at most three specified vertices $d$, $t$, and~$s$,
\item if $d$ exists, then it has degree $3$, $4$, or $5$, is oriented, and is on the boundary of~$F_G$,
\item if $t$ or $s$ exists, then it has degree $3$ and is on the boundary of~$F_G$,
\item $d$ has degree at most $5-a$ where $a$ is the number of unoriented degree $3$ vertices in~$G$,
\item $G$ has at most three $3$-edge-cuts, which can only be $\delta(d)$, $\delta(t)$, and $\delta(s)$, and
\item every vertex not in the boundary of $F_G$ has $5$ edge-disjoint paths to the boundary of~$F_G$.
\end{enumerate}

A \emph{3DTS graph} is a graph $G$ with the above definition, where (6) is replaced by 
\begin{enumerate}
\item[6'.] all vertices other than $d$, $t$, and $s$ have degree at least $4$, and if $d$, $t$, and $s$ all exist, then every $3$-edge-cut in $G$ separates one of $d$, $t$, and $s$ from the other two.
\end{enumerate}
\end{definition}

\begin{theorem}
 \label{3dts}All 3DTS graphs have a valid orientation.
 \label{3cutsdts}
 \label{maindts}
\end{theorem}

%

\subsection*{Circulant and Almost Circulant Graphs}

In the proof of the Strong $3$-Flow Conjecture for graphs in the projective plane, two specific classes of graphs will arise. Let $i$ be an odd integer greater than or equal to $5$, and define $B_i$ to be the circulant graph on $i$ vertices with jumps $1$ and $\frac{i-1}{2}$. Label the vertices $v_1,v_2,...,v_i$. Let $A_i$ be the graph on $i+1$ vertices obtained from $B_i$ by subdividing $v_1v_i$ with a vertex $v_0$, and adding an edge $v_0v_{\frac{i+1}{2}}$. We show that both classes of graphs satisfy the Strong $3$-Flow Conjecture. 

\begin{lemma}
\label{circul}
Let be an odd integer greater than or equal to $5$. Then $B_i$ and $A_i$, along with valid prescription functions, have a valid orientation. 
\end{lemma}
\begin{proof}
We obtain a valid orientation by lifting the pair of edges $v_1v_2$ and $v_1v_{\frac{i+3}{2}}$, and directing and deleting the edges incident with the following vertices in order, to meet each prescription:
$$v_1,(v_0),v_i,v_{\frac{i+1}{2}},v_{\frac{i-1}{2}},v_{i-1},v_{\frac{i-3}{2}},v_{i-2},...,v_{\frac{i+5}{2}},v_2.$$
This directs the edge $v_1v_2$, which in turn directs the edge $v_1v_{\frac{i+3}{2}}$. Since $v_{\frac{i+3}{2}}$ cannot be the only vertex whose prescription is not met, the graph has a valid orientation. 
\end{proof}

\section{Projective Plane}

\label{proj}

Here we prove that the Strong $3$-Flow Conjecture holds for graphs embedded in the projective plane. 

\begin{definition}
A \emph{PT graph} is a graph $G$ embedded in the projective plane, together with a valid prescription function $p:V(G)\rightarrow\{-1,0,1\}$, such that:
\begin{enumerate}
\item $G$ is $3$-edge-connected,
\item $G$ has a specified face $F_G$, and at most one specified vertex~$t$,
\item if $t$ exists, then it has degree $3$ and is in the boundary of~$F_G$,
\item $G$ has at most one $3$-edge-cut, which can only by $\delta(\{t\})$, and
\item every vertex not in the boundary of $F_G$ has $5$ edge-disjoint paths to the boundary of~$F_G$.
\end{enumerate}
We define all $3$-edge-connected graphs on at most two vertices to be PT graphs, regardless of vertex degrees. \\

A \emph{3PT graph} is a graph $G$ with the above definition, where (4) is replaced by 
\begin{enumerate}
\item[4'.] all vertices aside from $t$ have degree at least $4$, and if $t$ exists and $\delta(A)$ is a $3$-edge-cut with $t\in A$, then $A$ is contained in an open disk.
\end{enumerate}
\end{definition}

\begin{theorem}
All PT graphs have a valid orientation.
\label{pmain}
\end{theorem}

\begin{proof}
Let $G$ be a minimal counterexample with respect to the number of edges. If $|E(G)|=0$, then $G$ consists of only an isolated vertex, and so has a trivial valid orientation. Thus we may assume $G$ has at least one edge. \\

We will establish the following series of properties of $G$. 
\begin{description}
\item[PT1:] The graph $G$ does not contain a loop, unoriented parallel edges, or a cut vertex.
\item[PT2:] The face $F_G$ is bounded by a cycle. 
\item[PT3:] There is no contractible chord of $F_G$ incident with a degree $3$ or $4$ vertex. 
\end{description}
We define a \emph{Type 1 cut} to be an edge-cut $\delta(A)$ that does not intersect the boundary of $F_G$. We define a \emph{Type 2 cut} to be an edge-cut $\delta(A)$ that has exactly two edges in the boundary of $F_G$. We define all remaining edge-cuts (with at least $4$ edges in the boundary of $F_G$) to be of \emph{Type 3}.
\begin{description}
\item[PT4:] 
\begin{enumerate}[topsep=0pt]
\item The graph $G$ does not contain a $2$-robust at most $4$-edge-cut.
\item If $G$ contains a $2$-robust at most $5$-edge-cut $\delta(A)$, then it is either 
\begin{enumerate}
\item of Type 1, where $A$ is contained in an open disk and the boundary of $F_G$ is in $A$, or
\item of Type 2, where $A$ is contained in an open disk, $t\in A$, and $t$ does not have an incident edge in $\delta(A)$. 
\end{enumerate}
\item If $G$ has a $2$-robust Type 1 $6$-edge-cut $\delta(A)$ where $A$ is contained in an open disk, then the boundary of $F_G$ is in $A$. 
\end{enumerate}
\item[PT5:] The vertex $t$ exists. 
\end{description}
Let $u$ and $v$ be the boundary vertices adjacent to $t$, and let $w$ be the remaining vertex adjacent to~$t$.
\begin{description}
\item[PT6:] Vertices $u$ and $v$ have degree $4$.
\item[PT7:] Vertex $w$ has degree $5$.
\item[PT8:] The edge $tw$ is a chord. 
\item[PT9:] Both $u$ and $v$ are adjacent to $w$. 
\end{description}

The verification of these properties forms the bulk of the proof of Theorem \ref{pmain}. Property \ref{basics} is straightforward, and the proof is omitted. It can be read in \citep{thesis}. \\

\begin{Prop}{PT1}
The graph $G$ does not contain a loop, parallel edges, or a cut vertex. 
\label{basics}
\end{Prop}

\begin{Prop}{PT2}
The face $F_G$ is bounded by a cycle.
\label{cycle}
\end{Prop}
\begin{proof}
Suppose that $F_G$ is not bounded by a cycle. Then there exists a vertex $v$ that appears twice on the boundary walk of~$F_G$. By \ref{basics}, $v$ is not a cut vertex. As shown in Section \ref{projprelim}, $G$ is a planar graph with two specified faces, each containing $v$, and $G$ has at most one vertex of degree~$3$. Hence $G$ is an FT graph and has a valid orientation by Theorem \ref{mainft}, a contradiction. 
\end{proof}


\begin{claim}
The graph $G$ has no $2$-robust Type 1 cut $\delta(A)$ of size at most $6$ where $A$ is contained in an open disk and the boundary of $F_G$ is in~$G-A$. 
\label{type1}
\end{claim}
\begin{proof}
Suppose such a $\delta(A)$ exists in~$G$. Then by definition, $\delta_G(A)$ is at least a $5$-edge-cut. Let $G'$ be the graph obtained from $G$ by contracting $A$ to a vertex. It is clear that $G'$ is a PT graph, and thus has a valid orientation by the minimality of~$G$. Transfer this orientation to $G$. Let $G''$ be the graph obtained from $G$ by contracting $G-A$ to a vertex. Then $G''$ is a graph where all vertices have degree at least $5$, with a directed vertex $d'$ of degree $5$ or~$6$. Now $G[A]$ is planarly embedded by construction, where every vertex adjacent to $d'$ is on the outer face. Therefore, $G''$ can be embedded on the plane by inserting $d'$ into the outer face of $G[A]$. Choose the specified face to be incident with~$d'$. \\

If $d'$ has degree $5$, then $G''$ is a DTS graph and has a valid orientation by Theorem~\ref{maindts}. If $d'$ has degree $6$, delete one boundary edge incident with $d'$ to form a graph~$\bar{G}$. If $\bar{G}$ has a $2$-robust at most $3$-edge-cut, then $G$ has a $2$-robust at most $4$-edge-cut of Type 1, a contradiction. Thus $\bar{G}$ is a DTS graph and has a valid orientation by Theorem~\ref{maindts}. This yields a valid orientation of $G$, a contradiction. Hence no such cut exists. 
\end{proof}

\begin{claim}
The graph $G$ has no $2$-robust Type 2 cut $\delta(A)$ of size $4$. 
\label{4type2}
\end{claim}
\begin{proof}
Suppose such a cut $\delta(A)$ exists in~$G$. Either $A$ or $G-A$ is contained in an open disk. Without loss of generality, suppose that $A$ is contained in an open disk. Let $G'$ be the graph obtained from $G$ by contracting $A$ to a vertex. It is clear that $G'$ is a PT graph, and thus has a valid orientation by the minimality of~$G$. Transfer this orientation to~$G$. Let $G''$ be the graph obtained from $G$ by contracting $G-A$ to a vertex~$v$. Now $G[A]$ is planarly embedded by construction, where every vertex adjacent to $v$ is on the outer face. Therefore, $G''$ can be embedded on the plane by inserting $v$ into the outer face of~$G[A]$. Choose the specified face to be incident with $v$ and all vertices in $A$ incident with~$F_G$. \\

Then $G''$ is a planar graph with a specified face containing a directed vertex of degree~$4$. Hence $G''$ is a DTS graph and has a valid orientation by Theorem~\ref{maindts}. This yields a valid orientation of $G$, a contradiction. Hence no such cut exists. 
\end{proof}

\begin{claim}
The graph $G$ has no $2$-robust Type 2 cut of size $5$ where $A$ is contained in an open disk and $t\in G-A$. 
\label{type25}
\end{claim}
\begin{proof}
Suppose such a cut exists in~$G$. Let $G'$ be the graph obtained from $G$ by contracting $A$ to a vertex. It is clear that $G'$ is a PT graph, and thus has a valid orientation by the minimality of~$G$. Transfer this orientation to~$G$. Let $G''$ be the graph obtained from $G$ by contracting $G-A$ to a vertex. Then $G''$ can be embedded as a planar graph with a specified face containing a directed vertex of degree $5$ but no degree $3$ vertex. Hence $G''$ is a DTS graph and has a valid orientation by Theorem~\ref{maindts}. This yields a valid orientation of $G$, a contradiction. Hence no such cut exists. 
\end{proof}

In order to complete our analysis of Type 2 cuts, we must consider contractible chords. 

\begin{Prop}{PT3}
There is no contractible chord of $F_G$ incident with a degree $3$ or $4$ vertex. 
\label{chordspt}
\end{Prop}
\begin{proof}
Suppose that such a chord $e=uv$ exists where the degree of $u$ is $3$ or~$4$. Let $H$ and $K$ be subgraphs of $G$ such that $H\cap K=\{\{u,v\},\{e\}\}$, $H\cup K=G$, and $K$ is contained in an open disk. \\

Both $\delta(H)$ and $\delta(K)$ are $2$-robust, else $G$ has unoriented parallel edges. Both cuts are of Type~2. Therefore, $|\delta(H)|, |\delta(K)|\geq 5$, and so $deg(v)\geq 8$. \\

It is clear that $G/K$ is a PT graph and has a valid orientation by the minimality of~$G$. Transfer this orientation to $G$, and orient~$u$. Add a directed edge $e$ from $u$ to $v$ in $K$ (in the boundary of~$F_K$). Then $K+e$ can be embedded as a planar graph with a single specified face, a directed vertex of degree $3$ or $4$, and one other possible degree $3$ vertex~$t$. Hence $K+e$ is a DTS graph and has a valid orientation by Theorem~\ref{maindts}. This leads to a valid orientation of $G$, a contradiction. 
\end{proof}

\begin{claim}
The graph $G$ has no $2$-robust Type 2 cut of size $5$ where $A$ is contained in an open disk; $t\in A$; and $\delta(t)\cap\delta(A)$ is a single edge that is a boundary edge of~$F_G$. 
\label{type252}
\end{claim}
\begin{proof}
Suppose such a cut exists in~$G$. Let $G'$ be the graph obtained from $G$ by contracting $A$ to a vertex. Since $G'$ is a PT graph, it has a valid orientation by the minimality of~$G$. Transfer this orientation to~$G$. Let $G''$ be the graph obtained from $G$ by contracting $G-A$ to a vertex. Then $G''$ can be embedded as a planar graph with a specified face containing a directed vertex of degree $5$ adjacent to~$t$. Since $A$ is contained in an open disk, $t$ is not incident with a chord by \ref{chordspt}. Let $\hat{G}$ be the graph obtained from $G''$ by orienting and deleting~$t$. Then $\hat{G}$ has a directed vertex of degree $4$ and at most one vertex of degree~$3$. If $\hat{G}$ contains a $2$-robust at most $3$-edge-cut, then $G$ contains an at most $4$-edge-cut of Type 1 or 2, a contradiction. Hence $\hat{G}$ is a DTS graph and has a valid orientation by Theorem~\ref{maindts}. This yields a valid orientation of $G$, a contradiction. Hence no such cut exists. 
\end{proof}

\begin{claim}
The graph $G$ has no $2$-robust Type 3 cut of size at most $5$. 
\label{type3}
\end{claim}
\begin{proof}
Suppose that such a cut exists. Without loss of generality, suppose that $t\in A$. Let $G'$ be the graph obtained from $G$ by contracting $A$ to a single vertex. Since $G'$ is a PT graph, it has a valid orientation by the minimality of~$G$. Transfer this orientation to~$G$. Let $G''$ be the graph obtained from $G$ by contracting $G-A$ to a single directed vertex $d$ of degree $4$ or~$5$. Then $d$ appears twice in the boundary walk of~$F_{G''}$. As in the proof of \ref{cycle}, $G''$ is a planar graph with two specified faces that both contain the vertex~$d$. Since $t\in A$, $G''$ has no degree $3$ vertex. Therefore $G''$ is an FT graph and has a valid orientation by Claim~\ref{mainft}. This leads to a valid orientation of $G$, a contradiction. Hence no such cut exists. 
\end{proof}

We summarise the reducible/irreducible cuts in $G$.

\begin{Prop}{PT4}\mbox{}
\begin{enumerate}[topsep=0pt]
\label{pt4}
\item The graph $G$ does not contain a $2$-robust at most $4$-edge-cut.
\item If $G$ contains a $2$-robust at most $5$-edge-cut $\delta(A)$, then it is either 
\begin{enumerate}
\item of Type 1, where $A$ is contained in an open disk and the boundary of $F_G$ is in $A$, or
\item of Type 2, where $A$ is contained in an open disk, $t\in A$, and $t$ does not have an incident edge in $\delta(A)$. 
\end{enumerate}
\item If $G$ has a $2$-robust Type 1 $6$-edge-cut $\delta(A)$ where $A$ is contained in an open disk, then the boundary of $F_G$ is in $A$. 
\end{enumerate}
\end{Prop}
\begin{proof}\mbox{}
\begin{enumerate}[topsep=0pt]
\item By definition, $G$ has no $2$-robust either at most $3$-edge-cut, or Type 1 $4$-edge-cut. Claims \ref{4type2} and \ref{type3} show that $G$ has no $2$-robust $4$-edge-cut of Type 2 or~3. 
\item This is implied by Claims \ref{type1}, \ref{type25}, \ref{type252}, and~\ref{type3}. 
\item This is given by Claim~\ref{type1}. \qedhere
\end{enumerate}
\end{proof}\medskip

Figure \ref{cuts} shows the $2$-robust at most $5$-edge-cuts that exist in~$G$. The dashed line is used to represent the crosscap. We wish to reduce at low degree vertices in~$G$. We establish the existence of $t$ and consider its adjacent vertices. 

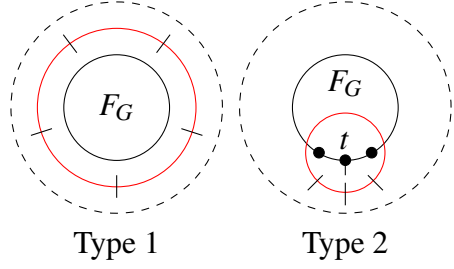
\begin{figure}
\begin{center}
\subfloat{
\begin{tikzpicture}
\draw (0,0) circle (20pt) node {$F_G$};
\draw[red] (0,0) circle (30pt);
\draw[dashed] (0,0) circle (40pt);
\draw (0,-0.9) -- (0,-1.2);
\draw (-0.86,-0.28) -- (-1.14,-0.37);
\draw (-0.53,0.73) -- (-0.71,0.97);
\draw (0.53,0.73) -- (0.71,0.97);
\draw (0.86,-0.28) -- (1.14,-0.37);
\draw (0,-1.5) circle (0pt) node[below] {Type 1};
\end{tikzpicture}
}
\subfloat{
\begin{tikzpicture}
\draw (0,0) circle (20pt) node[above] {$F_G$};
\draw[red] (0,-0.6) circle (15pt);
\draw[dashed] (0,0) circle (40pt);
\filldraw (0,-0.7) circle (2pt) node[above] {$t$};
\filldraw (0.35,-0.6) circle (2pt);
\filldraw (-0.35,-0.6) circle (2pt);
\draw (0,-0.7) -- (0,-0.9);
\draw (0,-1) -- (0,-1.3);
\draw (0.3,-0.9) -- (0.5,-1.1);
\draw (-0.3,-0.9) -- (-0.5,-1.1);
\draw (0,-1.5) circle (0pt) node[below] {Type 2};
\end{tikzpicture}
}
\caption{The $2$-robust at most $5$-edge-cuts that can exist in $G$.}
\label{cuts}
\end{center}
\end{figure}


\begin{Prop}{PT5}
The vertex $t$ exists.
\label{ptgraph}
\end{Prop}
\begin{proof}
Suppose that $t$ does not exist. We prove the following claims:
\begin{enumerate}
\item[a.] Every vertex on the boundary of $F_G$ has degree $4$. 
\item[b.] All vertices in $G$ are on the boundary of $F_G$. 
\item[c.] We have $G=B_k$, where $k\geq 5$ is odd.
\end{enumerate}
These properties provide the necessary structure to obtain a contradiction. 

\begin{Inclaim}{PT5a}
Every vertex on the boundary of $F_G$ has degree $4$. 
\end{Inclaim}
\begin{proof}
Consider the case where the boundary of $F_G$ contains a vertex $v$ of degree at least~$5$. Orient and delete a boundary edge incident with $v$, and call the resulting graph~$G'$. Then $G'$ has at most one degree $3$ vertex. Suppose that $G'$ contains a $2$-robust at most $3$-edge-cut. Then $G$ contains a $2$-robust at most $4$-edge-cut, a contradiction. Thus $G'$ is a PT graph. By the minimality of $G$, $G'$ has a valid orientation. This yields a valid orientation of~$G$, a contradiction. 
\end{proof}

\begin{Inclaim}{PT5b}
All vertices in $G$ are on the boundary of $F_G$. 
\end{Inclaim}
\begin{proof}
Suppose there exists a vertex $v$ on the boundary of $F_G$ that has an adjacent vertex $u$ not on the boundary of~$F_G$. Then $deg_G(u)\geq 5$. Let $e_1, e_2, e_3, e_4$ be the edges incident with $v$ in order, where $e_1$ and $e_4$ are on the boundary of $F_G$, and $e_2$ is incident with~$u$. Lift the pair of edges $e_3$, $e_4$, orient the remaining two edges incident with $v$ to satisfy $p(v)$, and delete $v$, calling the resulting graph~$G'$. Then $G'$ has at most one degree three vertex (the other endpoint of $e_1$). If $G'$ contains a $2$-robust at most $3$-edge-cut $\delta_{G'}(A)$, then $G$ contains a $2$-robust at most $4$-edge-cut, or a $2$-robust $5$-edge-cut that uses a boundary edge of $F_G$ and thus is of Type 2 or Type 3, a contradiction (since $t$ does not exist). Thus $G'$ is a PT graph, and so by the minimality of $G$, $G'$ has a valid orientation. This yields a valid orientation of~$G$, a contradiction.  
\end{proof}

Thus all vertices lie on the boundary of~$F_G$ and have degree $4$. Let $v_1,v_2,...,v_k$ be the vertices on the boundary of $F_G$ in order. 

\begin{Inclaim}{PT5c}
We have $G=B_k$.
\end{Inclaim}
\begin{proof}
Consider a vertex~$v_j$. It is clear from the construction that $v_j$ is adjacent to $v_{j-1}$ and~$v_{j+1}$. Let $v_a$ and $v_b$ be the remaining two vertices adjacent to~$v_j$. Suppose that $v_a$ and $v_b$ are not adjacent. Then $\delta_G(\{v_{a+1},v_{a+2},...,v_{b-1}\})$ is an at most $4$-edge-cut of Type~2. If it is not a $2$-robust $4$-edge-cut, then $G$ contains parallel edges, a contradiction. Hence $G$ contains a $2$-robust $4$-edge-cut of Type 2, a contradiction. Thus $v_a$ and $v_b$ are adjacent. \\

Without loss of generality, assume that $j-b>j-a$. Let 
$$S=\{v_a,v_{a+1},v_{a+2},...,v_{j-1}\},$$
$$T=\{v_{j+1},v_{j+2}...,v_{b-1},v_b\}.$$
If there exists an edge not on the boundary of $F_G$ with both endpoints in $S$, then either $G$ contains parallel edges, or $G$ contains a contractible chord incident with a degree $4$ vertex, a contradiction. The same is true of edges with both endpoints in~$T$. Hence every edge not in the boundary of $F_G$ and not incident with $v_j$ has one endpoint in $S$ and the other endpoint in~$T$. Since every vertex has degree $4$, $|S|=|T|$. If $k$ is even, then the non-contractible chords incident with a vertex are parallel edges. Hence $k$ is odd and $G=B_k$.
\end{proof}

By Lemma \ref{circul}, $G$ has a valid orientation. 
%
\end{proof}

Let $u$, $v$, and $w$ be the vertices adjacent to $t$, where $tu$ and $tv$ are on the boundary of~$F_G$. Since $G$ has no parallel edges, these three vertices are distinct. Note that \ref{chordspt} implies that $t$ is not incident with a contractible chord, so either $w$ is not in the boundary of $F_G$, or $tw$ is a non-contractible chord. 

\begin{claim}
Let $G'$ be a 3PT graph where $|E(G')|<|E(G)|$. Then $G'$ has a valid orientation. 
\label{3cutspt}
\end{claim}
\begin{proof}
Let $G'$ be a minimal counterexample. If $G'$ is a PT graph, then $G'$ has a valid orientation by the minimality of~$G$. Thus we may assume that $G'$ has a $2$-robust $3$-edge-cut $\delta_{G'}(A)$ where $t\in A$. By the definition of a 3PT graph, $A$ is contained in an open disk. \\

Let $G_1$ be the graph obtained from $G'$ by contracting $A$ to a vertex. Then any $3$-edge-cut in $G_1$ is a $3$-edge-cut in $G'$, so $G_1$ is a 3PT graph and has a valid orientation by the minimality of~$G'$. Transfer this orientation to~$G'$. Let $G_2$ be the graph obtained from $G'$ by contracting $G'-A$ to a vertex. Then $G_2$ is a 3DTS graph and has a valid orientation by Lemma~\ref{3dts}. This leads to a valid orientation of $G'$, a contradiction. 
\end{proof}

\begin{claim}
At least two of $u$, $v$, and $w$ have degree $4$. 
\label{two} 
\end{claim}
\begin{proof}
Suppose not. Let $G'$ be the graph obtained from $G$ by orienting and deleting~$t$. Then $G'$ contains at most one vertex of degree~$3$. Suppose that $G'$ contains a $2$-robust at most $3$-edge-cut~$\delta_{G'}(A)$. Then $G$ contains a $2$-robust at most $4$-edge-cut, a contradiction.  Hence $G'$ is a PT graph. Thus by the minimality of $G$, $G'$ has a valid orientation. This yields a valid orientation of $G$, a contradiction. \end{proof}

\begin{claim}
Vertex $w$ does not have degree $4$.
\label{w4}
\end{claim}
\begin{proof}
Suppose that $deg(w)=4$. Note that $tw$ is a non-contractible chord of~$F_G$. Let $g_1$, $g_2$, $g_t$, $g_3$ be the edges incident with $w$ in order, where $g_1$ and $g_3$ are on the boundary of $F_G$, and $g_t=wt$. We may choose the labelling of $u$ and $v$ so that the $uw$-path in the boundary of $F_G$ contains $g_1$ but not~$t$. \\

Suppose that $g_3=wv$. Then $G$ can be redrawn with $tw$ inside $F_G$ to yield an FT graph with $t$ and $w$ on both specified faces. Figure \ref{pttwo} shows this drawing. By Theorem \ref{mainft}, $G$ has a valid orientation, a contradiction. \\

We may now assume that $g_3\neq wv$. Lift the pair $g_1$, $g_2$, and orient the remaining edges incident with $w$ to satisfy~$p(w)$. Orient the remaining edges incident with $t$ to satisfy~$p(t)$. Delete $w$ and $t$, and call the resulting graph~$G'$. Then $G'$ is a planar graph. There are three possible degree three vertices in $G'$: $u$, $v$, and the vertex incident with $g_3$ (which by assumption are distinct). If $G'$ contains a $2$-robust at most $2$-edge-cut, then $G$ contains a $2$-robust at most $4$-edge-cut, a contradiction to \ref{pt4}. \\

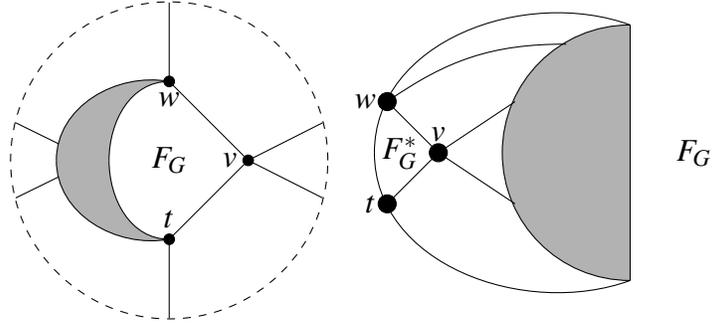
\begin{figure}
\begin{center}
\subfloat{
\begin{tikzpicture}
\fill[white!70!black] (0,1.05) to[out=180,in=90] (-0.8,0) to[out=-90,in=180] (0,-1.05) to[out=190,in=-90] (-1.5,0) to[out=90,in=170] (0,1.05);
\draw[dashed] (0,0) circle (60pt);
\filldraw (1.05,0) circle (2pt) node[left] {$v$};
\filldraw (0,-1.05) circle (2pt) node[above] {$t$};
\filldraw (0,1.05) circle (2pt) node[below] {$w$};
\filldraw (0,0) circle (0pt) node {$F_G$};
\draw (1.05,0) -- (0,1.05);
\draw (1.05,0) -- (0,-1.05);
\draw (0,1.05) to[out=180,in=90] (-0.8,0);
\draw (0,-1.05) to[out=180,in=-90] (-0.8,0);

\draw (0,1.05) to[out=170,in=90] (-1.5,0);
\draw (0,-1.05) to[out=190,in=-90] (-1.5,0);

\draw (0,-1.05) -- (0,-2.1);
\draw (0,1.05) -- (0,2.1);

\draw (1.05,0) -- (2.04,0.5);
\draw (1.05,0) -- (2.04,-0.5);

\draw (-2.04,0.5) -- (-1.47,0.22);
\draw (-2.04,-0.5) -- (-1.47,-0.22);
\end{tikzpicture}
}
\subfloat{
\begin{tikzpicture}[scale = 1.7]
\fill[white!70!black] (1,1) -- (1,-1) to[out=180,in=-90] (0,0) to[out=90,in=180] (1,1);

\draw (1,1) -- (1,-1);
\draw (1,1) to[out=180,in=90] (0,0);
\draw (1,-1) to[out=180,in=-90] (0,0);
\draw (1,1) to[out=160,in=90] (-1,0);
\draw (1,-1) to[out=-160,in=-90] (-1,0);

\filldraw (-0.9,0.4) circle (2pt) node[left] {$w$};
\filldraw (-0.9,-0.4) circle (2pt) node[left] {$t$};
\filldraw (-0.5,0) circle (2pt) node[above] {$v$};

\draw (-0.9,0.4) -- (-0.5,0);
\draw (-0.9,-0.4) -- (-0.5,0);
\draw (-0.5,0) -- (0.1,-0.4);
\draw (-0.5,0) -- (0.1,0.4);

\draw (-0.9,0.4) to[out=35,in=180] (0.5,0.85);

\draw (1.5,0) circle (0pt) node {$F_G$};
\draw (-0.8,0) circle (0pt) node {$F_G^*$};
\end{tikzpicture}
}
\caption{Claim \ref{w4}: Planar drawing of $G$.}
\label{pttwo}
\end{center}
\end{figure}
%

If all $2$-robust $3$-edge-cuts in $G'$ separate two degree $3$ vertices, then by Theorem~\ref{3cutsdts}, $G'$ has a valid orientation that leads to a valid orientation of $G$, a contradiction. Suppose that $G'$ contains a $2$-robust $3$-edge-cut $\delta_{G'}(A)$ that does not separate the degree $3$ vertices. We assume that $\delta_G(A)$ is not an at most $4$-edge-cut. Then $\delta_G(A)$ is a $5$-edge-cut using edges $g_1$ and~$g_2$. This is a Type 2 cut, for which $t$ is in the side not contained in an open disk, a contradiction. Figure \ref{ptthree} shows this cut. Orient a degree three vertex in $G'$ to obtain a DTS graph. By Theorem \ref{maindts}, $G'$ has a valid orientation. This extends to a valid orientation of $G$, a contradiction. 
\end{proof}

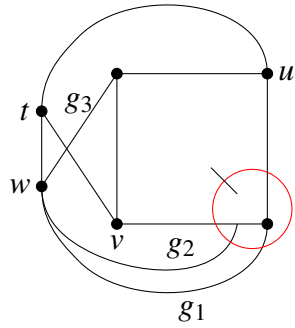
\begin{figure}
\begin{center}
\begin{tikzpicture}
\draw (0,-2) circle (0pt);

\filldraw (1,1) circle (2pt) node[right] {$u$};
\filldraw (-1,1) circle (2pt);
\filldraw (1,-1) circle (2pt);
\filldraw (-1,-1) circle (2pt) node[below] {$v$};
\draw (1,1) -- (-1,1);
\draw (1,1) -- (1,-1);
\draw (1,-1) -- (-1,-1);
\draw (-1,1) -- (-1,-1);

\filldraw (-2,-0.5) circle (2pt) node[left] {$w$};
\filldraw (-2,0.5) circle (2pt) node[left] {$t$};
\draw (-2,-0.5) -- (-2,0.5);

\draw (-2,-0.5) -- node[above=2pt, midway] {$g_3$} (-1,1);
\draw (-2,0.5) -- (-1,-1);
\draw (-2,0.5) to[out=90,in=-135] (-1.5,1.5);
\draw (1,1) to[out=90,in=45] (-1.5,1.5);
\draw (-2,-0.5) to[out=-90,in=135] (-1.5,-1.5);
\draw (1,-1) to[out=-90,in=-45] node[below, midway] {$g_1$} (-1.5,-1.5);
\draw (-2,-0.5) to[out=-90,in=155] (-1.3,-1.4);
\draw (0.6,-1) to[out=-100,in=-25] node[above, midway] {$g_2$} (-1.3,-1.4);

\draw[red] (0.8,-0.8) circle (15pt);
\draw (0.6,-0.6) -- (0.25,-0.25);
\end{tikzpicture}
\caption{Claim \ref{w4}: $2$-robust Type 2 $5$-edge-cut.}
\label{ptthree}
\end{center}
\end{figure}

\begin{Prop}{PT6}
Vertices $u$ and $v$ have degree $4$. 
\label{uorv}
\end{Prop}
\begin{proof}
This is an immediate corollary of Claims \ref{two} and \ref{w4}.
 \end{proof}



\begin{Prop}{PT7}
Vertex $w$ has degree $5$.
\label{w5}
\end{Prop}
\begin{proof}
Suppose that $deg(w)\geq 6$. Let $e_1$, $e_2$, $e_3$, $e_t$ be the edges incident with $u$ in order, where $e_1$ is on the boundary of $F_G$, and $e_t=ut$. We prove the following claims
\begin{enumerate}
\item[a.] Edge $e_3$ is incident with two vertices of degree $4$. 
\item[b.] Edge $tw$ is not a chord. 
\end{enumerate}
These provide the necessary structure to complete the proof.

\begin{Inclaim}{PT7a}
Edge $e_3$ is incident with two vertices of degree $4$. 
\end{Inclaim}
\begin{proof}
Suppose that $e_3$ is not incident with two vertices of degree~$4$. Lift the pair $e_1$, $e_2$, and orient the remaining edges incident with~$u$. Orient the remaining edges incident with $t$, and delete $u$ and $t$, calling the resulting graph~$G'$. Then $G'$ contains at most one vertex of degree $3$, which is~$v$. \\

We now check that $G'$ has no $2$-robust $2$- or $3$-edge-cuts. The cases are indicated in italics. We follow a similar process in most future cases. 
\emph{Suppose that $G'$ contains a $2$-robust at most $2$-edge-cut~$\delta_{G'}(A)$.} Then $G$ contains a $2$-robust at most $4$-edge-cut, a contradiction. \emph{Suppose that $G'$ contains a $2$-robust $3$-edge-cut~$\delta_{G'}(A)$.} Then $G$ contains a corresponding $2$-robust cut $\delta(C)$ of size at most~$5$. There are two options. First, $\delta_G(C)$ is of Type 2 and has $t$ on the side that is contained in an open disk. We assume that this side is~$C$. By Claim \ref{3cutspt}, $v\not\in C$. Hence $tv\in\delta_G(C)$. We have $u,w\in C$, else a smaller cut exists. We conclude that $e_3\in\delta_G(C)$. Let $B$ be the maximal connected subgraph of $C$ containing $w$ but not~$t$. Then $\delta_G(B)$ is an at most $4$-edge-cut of Type 1, a contradiction. This cut can be seen in Figure~\ref{ptfour}. Second, $\delta_G(C)$ is of Type 1 in~$G$. Then $\delta_G(C)$ has $v$ on the side contained in a disk in $G'$, and thus $G'$ has a valid orientation by Claim~\ref{3cutspt}. This leads to a valid orientation of $G$, a contradiction.\\ 

Hence $G'$ is a PT graph, and thus has a valid orientation by the minimality of~$G$. This extends to a valid orientation of $G$, a contradiction. 
\end{proof}

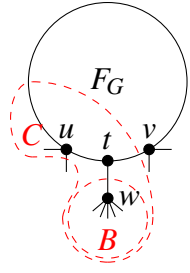
\begin{figure}
\begin{center}
\begin{tikzpicture}
\draw (0,0) circle (30pt) node {$F_G$};
\filldraw (0,-1.05) circle (2pt) node[above] {$t$};
\filldraw (0.55,-0.9) circle (2pt) node[above] {$v$};
\filldraw (-0.55,-0.9) circle (2pt) node[above] {$u$};

\draw (-0.55,-0.9) -- (-0.55,-1.2);
\draw (-0.55,-0.9) -- (-0.85,-0.9);
\draw (0.55,-0.9) -- (0.55,-1.2);
\draw (0.55,-0.9) -- (0.85,-0.9);

\filldraw (0,-1.55) circle (2pt) node[right] {$w$};
\draw (0,-1.05) -- (0,-1.55);
\draw (0,-1.55) -- (0,-1.82);
\draw (0,-1.55) -- (0.1,-1.8);
\draw (0,-1.55) -- (-0.1,-1.8);
\draw (0,-1.55) -- (0.2,-1.72);
\draw (0,-1.55) -- (-0.2,-1.72);

\draw[red,dashed] (0,-1.82) circle (15pt) node[below] {$B$};

\draw[red,dashed] (-1,0) to[out=0,in=90] (0.6,-1.8);
\draw[red,dashed] (0.6,-1.8) to[out=-90,in=0] (0,-2.4);
\draw[red,dashed] (0,-2.4) to[out=180,in=-90] (-0.6,-1.8);
\draw[red,dashed] (-0.6,-1.8) to[out=80,in=-80] (-0.4,-1.2);
\draw[red,dashed] (-0.4,-1.2) to[out=100,in=0] (-1,-1) node[above] {$C$};
\draw[red,dashed] (-1,-1) to[out=180,in=180] (-1,0);
\end{tikzpicture}
\caption{\ref{w5}: $2$-robust Type 1 $4$-edge-cut.}
\label{ptfour}
\end{center}
\end{figure}

%

Since $e_3$ is incident with two vertices of degree~$4$, $e_3$ is a non-contractible chord. Similarly, $v$ must have an analogous incident chord. Neither chord is incident with $w$, since $w$ has degree at least $6$. Let $x$ and $y$ be the vertices adjacent to $u$ and $v$ respectively via a chord. \\

\begin{Inclaim}{PT7b}
Edge $tw$ is not a chord. 
\end{Inclaim}
\begin{proof}
Suppose that $tw$ is a chord. Let $B$ be the set of vertices in the interior of the closed disc bounded by $wt$, $tu$, $ux$, and the $wx$-subpath of $F_G-t$. If $B$ is empty, consider the same set with respect to $v$ and~$y$. Now $\delta_G(B)$ contains at most $2$ edges incident with~$x$. Hence it contains at least two edges incident with $w$, since $\delta(t)$ is the only $3$-edge-cut in~$G$. Since $G$ has no parallel edges, $\delta_G(B)$ is $2$-robust. This cut can be seen in Figure~\ref{ptfive}. Let $C$ be the minimal $2$-robust edge-cut where $C\subseteq B$ and $\delta_G(C)$ contains at most three edges not incident with~$w$. \\

Contract $C$ to a vertex, calling the resulting graph~$G'$. Then $G'$ is a PT graph and has a valid orientation by the minimality of~$G$. Transfer this orientation to $G$ and contract $G-C$, calling the resulting graph~$G'$. Delete edges incident with $w$ to make the vertex of contraction a degree $4$ vertex. Since $G$ has no parallel edges, at most one degree $3$ vertex results from this process. If $G''$ has a $2$-robust at most $3$-edge-cut, then $C$ was not minimal, a contradiction. 
\end{proof}

\begin{figure}
\begin{center}
\begin{tikzpicture}
\draw (0,0) circle (30pt) node {$F_G$};
\filldraw (0,-1.05) circle (2pt) node[above] {$t$};
\filldraw (0.55,-0.9) circle (2pt) node[above] {$v$};
\filldraw (-0.55,-0.9) circle (2pt) node[above] {$u$};
\filldraw (0,1.05) circle (2pt) node[below] {$w$};
\filldraw (0.55,0.9) circle (2pt) node[below] {$x$};
\filldraw (-0.55,0.9) circle (2pt) node[below] {$y$};

\draw[dashed] (0,0) circle (60pt);
\draw (0,-1.05) -- (0,-2.1);
\draw (0,1.05) -- (0,2.1);
\draw (-0.55,-0.9) -- (-1.1,-1.79);
\draw (0.55,0.9) -- (1.1,1.79);
\draw (0.55,-0.9) -- (1.1,-1.79);
\draw (-0.55,0.9) -- (-1.1,1.79);
\draw (-0.55,-0.9) -- (-0.85,-0.9);
\draw (0.55,-0.9) -- (0.85,-0.9);

\draw (0.55,0.9) -- (0.4,1.12);
\draw (0,1.05) -- (0.15,1.29);
\draw (0,1.05) -- (0.2,1.25);
\draw (0,1.05) -- (0.25,1.2);
\draw (0,1.05) -- (0.28,1.15);

\draw[red, dashed] (0.3,1.2) circle (6pt);
\end{tikzpicture}
\caption{\ref{w5}: Small Type 1 or 2 cut.}
\label{ptfive}
\end{center}
\end{figure}
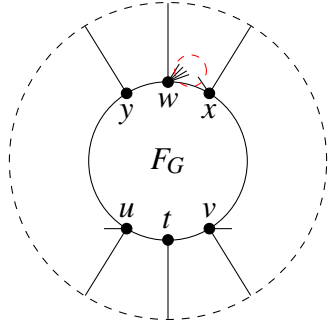


Since $tw$ is not a chord, $G$ contains a $2$-robust cut of size at most $5$, $\delta_G(A)$, using at most two edges incident with each $x$ and $y$, and the edge~$tw$. This cut can be seen in Figure~\ref{ptsix}. Then $t\in G-A$, and the graph obtained from contracting $G-A$ to a single vertex is planar. By Claim \ref{type25}, $G$ has a valid orientation. 
\end{proof}

\begin{figure}
\begin{center}
\begin{tikzpicture}
\draw (0,0) circle (30pt) node {$F_G$};
\filldraw (0,-1.05) circle (2pt) node[above] {$t$};
\filldraw (0.55,-0.9) circle (2pt) node[above] {$v$};
\filldraw (-0.55,-0.9) circle (2pt) node[above] {$u$};
\filldraw (0.55,0.9) circle (2pt) node[below] {$x$};
\filldraw (-0.55,0.9) circle (2pt) node[below] {$y$};

\draw[dashed] (0,0) circle (60pt);
\draw (-0.55,-0.9) -- (-1.1,-1.79);
\draw (0.55,0.9) -- (1.1,1.79);
\draw (0.55,-0.9) -- (1.1,-1.79);
\draw (-0.55,0.9) -- (-1.1,1.79);
\draw (-0.55,-0.9) -- (-0.85,-0.9);
\draw (0.55,-0.9) -- (0.85,-0.9);

\filldraw (0,-1.55) circle (2pt) node[right] {$w$};
\draw (0,-1.05) -- (0,-1.55);
\draw (0,-1.55) -- (0,-1.82);
\draw (0,-1.55) -- (0.1,-1.8);
\draw (0,-1.55) -- (-0.1,-1.8);
\draw (0,-1.55) -- (0.2,-1.72);
\draw (0,-1.55) -- (-0.2,-1.72);

\draw (0.55,0.9) -- (0.55,1.2);
\draw (-0.55,0.9) -- (-0.55,1.2);

\draw[red,dashed] (0,0.8) to[out=0,in=-115] (1,1.85);
\draw[red,dashed] (0,0.8) to[out=180,in=-65] (-1,1.85);
\draw[red,dashed] (0,-1.3) to[out=0,in=115] (1,-1.85);
\draw[red,dashed] (0,-1.3) to[out=180,in=65] (-1,-1.85);
\end{tikzpicture}
\caption{\ref{w5}: Small Type 2 cut.}
\label{ptsix}
\end{center}
\end{figure}
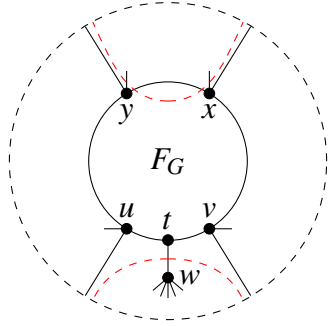


Therefore, $u$ and $v$ have degree $4$, and $w$ has degree~$5$. 

\begin{Prop}{PT8}
The edge $tw$ is a chord.
\label{twchord}
\end{Prop}
\begin{proof}
Suppose that $tw$ is not a chord. Let $e_1,e_2,e_3,e_t$ be the edges incident with $u$ in order, where $e_t=ut$, and $e_1$ is on the boundary of~$F_G$. We first prove the following property.

\begin{Inclaim}{PT8a}
Either $e_3$ is incident with two degree four vertices or $e_3$ is incident with~$w$.
\end{Inclaim}
\begin{proof}
Suppose that $e_3$ is not incident with two degree four vertices, and is not incident with~$w$. Lift the pair $e_1$, $e_2$, orient the remaining edges incident with $u$ to satisfy $p(u)$, and orient the remaining edges incident with $t$ to satisfy~$p(t)$. Delete $u$ and $t$, calling the resulting graph~$G'$. Then in $G'$, the only possible degree three vertex is~$v$. The analysis that $G'$ has no small cuts is equivalent to that in \ref{w5}. Hence $G'$ is a PT graph, and thus has a valid orientation by the minimality of~$G$. This extends to a valid orientation of $G$, a contradiction. 
\end{proof}



%

Therefore either $e_3$ is incident with two degree four vertices or $e_3$ is incident with~$w$. The same must be true of the corresponding edge incident with~$v$. There are three cases. 

\begin{enumerate}
\item First, suppose that both $u$ and $v$ are adjacent to~$w$. \emph{Suppose that $e_2$ is not incident with two degree four vertices.} Then orient and delete $e_1$ and $e_2$ to satisfy $p(u)$, and contract the set of vertices $\{u,t,v,w\}$ to a single degree four vertex, calling the resulting graph~$G'$. Then $G'$ has at most one degree three vertex, which is incident with $e_1$ in~$G$. If $G'$ contains a $2$-robust edge-cut of size at most $3$, then $G$ contains a $2$-robust edge-cut of size at most $5$ that contains a boundary edge incident with $t$, a contradiction. Hence $G'$ is a PT graph. By the minimality of $G$, $G'$ has a valid orientation. Transfer this orientation to~$G$. Orient the remaining two edges incident with $v$ to satisfy $p(v)$, and the remaining two edges incident with $t$ to satisfy~$p(t)$. Since $p(u)$ is satisfied by the orientations of $e_1$ and $e_2$, the direction of $e_3$ is determined to be the opposite (relative to $u$) of the direction of~$e_t$. Since $w$ cannot be the only vertex whose prescription is not met, this is a valid orientation of $G$, a contradiction. \\

\emph{We now assume that $e_2$ is incident with two degree four vertices.} The same must be true of the corresponding edge incident with~$v$. Let these vertices be $x$ and $y$ respectively. Suppose that $x$ and $y$ are not adjacent. Let $C_1$ and $C_2$ be the components of $G-\{x,y,w\}$, where the labelling is chosen so that $t\in C_2$. Consider the cut $\delta_G(C_1)$. If $C_1$ contains a single vertex, then $G$ has parallel edges, a contradiction. Hence $\delta_G(C_1)$ is $2$-robust. Note that $\delta_G(C_1)$ is a Type 2 cut and has size at most~$6$. This cut is shown in Figure~\ref{ptseven}.\\

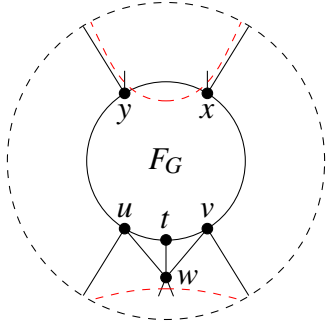
\begin{figure}
\begin{center}
\begin{tikzpicture}
\draw (0,0) circle (30pt) node {$F_G$};
\filldraw (0,-1.05) circle (2pt) node[above] {$t$};
\filldraw (0.55,-0.9) circle (2pt) node[above] {$v$};
\filldraw (-0.55,-0.9) circle (2pt) node[above] {$u$};
\filldraw (0.55,0.9) circle (2pt) node[below] {$x$};
\filldraw (-0.55,0.9) circle (2pt) node[below] {$y$};

\draw[dashed] (0,0) circle (60pt);
\draw (-0.55,-0.9) -- (-1.1,-1.79);
\draw (0.55,0.9) -- (1.1,1.79);
\draw (0.55,-0.9) -- (1.1,-1.79);
\draw (-0.55,0.9) -- (-1.1,1.79);

\filldraw (0,-1.55) circle (2pt) node[right] {$w$};
\draw (0,-1.05) -- (0,-1.55);
\draw (0,-1.55) -- (0.1,-1.8);
\draw (0,-1.55) -- (-0.1,-1.8);
\draw (-0.55,-0.9) -- (0,-1.55);
\draw (0.55,-0.9) -- (0,-1.55);

\draw (0.55,0.9) -- (0.55,1.2);
\draw (-0.55,0.9) -- (-0.55,1.2);

\draw[red,dashed] (0,0.8) to[out=0,in=-115] (1,1.85);
\draw[red,dashed] (0,0.8) to[out=180,in=-65] (-1,1.85);
\draw[red,dashed] (0,-1.7) to[out=0,in=165] (1,-1.85);
\draw[red,dashed] (0,-1.7) to[out=180,in=15] (-1,-1.85);
\end{tikzpicture}
\caption{\ref{twchord}: Small Type 2 cut (1).}
\label{ptseven}
\end{center}
\end{figure}


Contract $C_1$ to a vertex, calling the resulting graph $G'$. It is clear that $G'$ is a PT graph. Thus $G'$ has a valid orientation by the minimality of~$G$. Transfer this orientation to $G$ and contract $G-C_1$ calling the resulting graph~$G''$. Note that $G''$ can be embedded in the plane with a directed vertex $d'$ of degree at most~$6$. Delete the edges in the cut that are incident with $x$, calling the resulting graph~$\bar{G}$. Then $\bar{G}$ is planar, has at most one vertex of degree~$3$, and has a directed vertex of degree at most~$4$. If $\bar{G}$ has a $2$-robust at most $3$-edge-cut, then $G$ has a $2$-robust at most $5$-edge-cut, which is of Type 1, or Type 2 with $t$ in the side not in an open disk, a contradiction. Hence $\bar{G}$ is a DTS graph and has a valid orientation by Theorem~\ref{maindts}. This leads to a valid orientation of $G$, a contradiction. \\

Hence we may assume that $x$ and $y$ are adjacent. Note that $G-\{t,u,v,w,x,y\}$ has two components: $A$ with neighbours in $G$ among $x$, $y$, and $w$, and $B$, with neighbours in $G$ among $u$, $v$, $x$, and $y$. If $\delta_G(A)$ is a 2-robust cut, then it has size at most $4$ and is of Type 1, a contradiction. If $A$ contains a single vertex, then $G$ has parallel edges incident with $w$, a contradiction. Hence $x$ and $y$ are adjacent to~$w$. This graph is shown in Figure~\ref{pteight}. \\

\begin{figure}
\begin{center}
\begin{tikzpicture}
\draw (0,0) circle (30pt) node {$F_G$};
\filldraw (0,-1.05) circle (2pt) node[above] {$t$};
\filldraw (0.55,-0.9) circle (2pt) node[above] {$v$};
\filldraw (-0.55,-0.9) circle (2pt) node[above] {$u$};
\filldraw (0.55,0.9) circle (2pt) node[below] {$x$};
\filldraw (-0.55,0.9) circle (2pt) node[below] {$y$};

\draw[dashed] (0,0) circle (60pt);
\draw (-0.55,-0.9) -- (-1.1,-1.79);
\draw (0.55,0.9) -- (1.1,1.79);
\draw (0.55,-0.9) -- (1.1,-1.79);
\draw (-0.55,0.9) -- (-1.1,1.79);

\filldraw (0,-1.55) circle (2pt) node[right] {$w$};
\draw (0,-1.05) -- (0,-1.55);
\draw (-0.55,-0.9) -- (0,-1.55);
\draw (0.55,-0.9) -- (0,-1.55);

\draw (0.55,0.9) -- (0.3,2.08);
\draw (-0.55,0.9) -- (-0.3,2.08);
\draw (0,-1.55) -- (0.3,-2.08);
\draw (0,-1.55) -- (-0.3,-2.08);
\end{tikzpicture}
\caption{\ref{twchord}: $x$ and $y$ adjacent to $w$}
\label{pteight}
\end{center}
\end{figure}
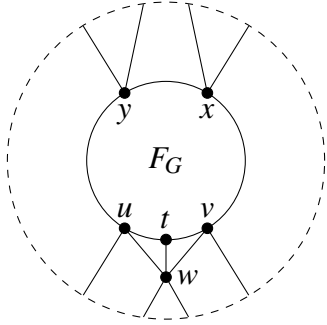


Let $A=\{t,u,v,w,x,y\}$. If $\delta_G(A)$ is a 2-robust cut, then it has size $4$ and is of Type 3, a contradiction. If $|V(G-A)|=1$, then this vertex is repeated in the boundary walk of $F_G$, so the boundary of $F_G$ is not a cycle, a contradiction. Hence $|V(G-A)|=0$. Lift the pair of edges $tv$, $vw$ and orient $v$, $x$, $y$, $u$, $t$ in order (each having at least two unoriented edges). This determines the direction of $vw$, and since $w$ cannot be the only vertex to not meet its prescription, the result is a valid orientation of $G$, a contradiction.

\item Now suppose that $e_3$ and the corresponding edge incident with $v$ each have two endpoints of degree $4$, $x$ and $y$ respectively. Then $x$ and $y$ are on the boundary of~$F_G$. Then $G$ contains a $2$-robust edge-cut $\delta_G(A)$ of size at most $5$ (containing $tw$, one or two edges incident with $x$, and one or two edges incident with~$y$). If the labelling is chosen so that $t\in G-A$, then the graph obtained by contracting $G-A$ to a single vertex is planar, so by Claim \ref{type25}, $G$ has a valid orientation. 

\item Finally, suppose without loss of generality that $e_3$ has two endpoints of degree $4$, and $v$ is adjacent to~$w$. Let $x$ be the other endpoint of~$e_3$. Let $f_1$ and $f_2$ be the edges incident with $v$ that are not $vt$ or $vw$, where $f_1$ is on the boundary of~$F_G$. \emph{Assume that $f_2$ does not have two incident vertices of degree~$4$.} Orient and delete $f_1$ and $f_2$ to satisfy $p(v)$, and contract the set of vertices $\{v,t,w\}$ to a single vertex of degree $4$, calling the resulting graph~$G'$. If $G'$ contains a $2$-robust at most $3$-edge-cut, then $G$ contains a $2$-robust at most $5$-edge-cut containing a boundary edge incident with $t$, a contradiction. Then $G'$ is a PT graph, so by the minimality of $G$, $G'$ has a valid orientation. Transfer this orientation to~$G$. Orient the remaining two edges incident with $t$ to satisfy~$p(t)$. Since $p(v)$ is satisfied by the orientations of $f_1$ and $f_2$, the direction of $vw$ is determined to be the opposite (relative to $v$) of the direction of~$vt$. Since $w$ cannot be the only vertex whose prescription is not met, this is a valid orientation of $G$, a contradiction. \\

\emph{We now assume that $f_2$ has two incident vertices of degree~$4$.} Then $f_2$ is a chord. Let $y$ be the other endpoint of~$f_2$. Consider the cut $\delta_G(A)$ where $x,y,t,v\in A$, $w\in G-A$, and $G-A$ is connected and maximised. If $G-A$ contains only one vertex, then $G$ has parallel edges, a contradiction. Hence $\delta_G(A)$ is $2$-robust. Note that $\delta_G(A)$ is a Type 2 cut and has size at most~$6$. This cut is shown in Figure~\ref{ptnine}.\\

\begin{figure}
\begin{center}
\begin{tikzpicture}
\draw (0,0) circle (30pt) node {$F_G$};
\filldraw (0,-1.05) circle (2pt) node[above] {$t$};
\filldraw (0.55,-0.9) circle (2pt) node[above] {$v$};
\filldraw (-0.55,-0.9) circle (2pt) node[above] {$u$};
\filldraw (0.55,0.9) circle (2pt) node[below] {$x$};
\filldraw (-0.55,0.9) circle (2pt) node[below] {$y$};

\draw[dashed] (0,0) circle (60pt);
\draw (-0.55,-0.9) -- (-1.1,-1.79);
\draw (0.55,0.9) -- (1.1,1.79);
\draw (0.55,-0.9) -- (1.1,-1.79);
\draw (-0.55,0.9) -- (-1.1,1.79);

\filldraw (0,-1.55) circle (2pt) node[right] {$w$};
\draw (0,-1.05) -- (0,-1.55);
\draw (-0.55,-0.9) -- (-0.85,-0.9);
\draw (0.55,-0.9) -- (0,-1.55);

\draw (0.55,0.9) -- (0.55,1.2);
\draw (-0.55,0.9) -- (-0.55,1.2);
\draw (0,-1.55) -- (0,-1.85);
\draw (0,-1.55) -- (0.3,-1.8);
\draw (0,-1.55) -- (-0.3,-1.8);

\draw[red,dashed] (0,0.8) to[out=0,in=-115] (1,1.85);
\draw[red,dashed] (0,0.8) to[out=180,in=-65] (-1,1.85);
\draw[red,dashed] (0,-1.3) to[out=0,in=115] (1,-1.85);
\draw[red,dashed] (0,-1.3) to[out=180,in=65] (-1,-1.85);
\end{tikzpicture}
\caption{\ref{twchord}: Small Type 2 cut (2).}
\label{ptnine}
\end{center}
\end{figure}
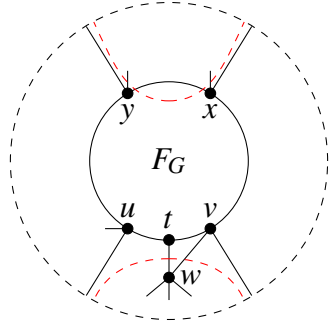


Contract $G-A$ to a vertex, calling the resulting graph~$G'$. It is clear that $G'$ is a PT graph. Thus $G'$ has a valid orientation by the minimality of~$G$. Transfer this orientation to $G$ and contract $A$ calling the resulting graph~$G''$. Note that $G''$ is planar and has a directed vertex $d'$ of degree at most~$6$. Delete two consecutive edges incident with $d'$ where one is a boundary edge, calling the resulting graph~$\bar{G}$. Then $\bar{G}$ is planar and has at most one vertex of degree~$3$. If $\bar{G}$ has a $2$-robust at most $3$-edge-cut, then $G$ has a $2$-robust at most $5$-edge-cut, either of Type 1, or Type 2 where $t$ is on the side not contained in an open disk, a contradiction. Hence $\bar{G}$ is a DTS graph and has a valid orientation by Theorem~\ref{maindts}. This leads to a valid orientation of $G$, a contradiction. \qedhere
\end{enumerate}
\end{proof}\medskip
%
%

We are left with the case where $tw$ is a chord. Let the edges incident with $u$ be $e_1,e_2,e_3,e_t$ in order, where $e_t=ut$, and $e_1$ is a boundary edge of~$F_G$. Suppose that $e_3$ is not incident with two degree four vertices, and is not incident with~$w$. Lift the pair $e_1$, $e_2$, orient the remaining edges incident with $u$ to satisfy $p(u)$, and orient the remaining edges incident with $t$ to satisfy~$p(t)$. Delete $u$ and $t$, calling the resulting graph~$G'$. Then in $G'$, the only degree three vertex is~$v$. The argument that $G'$ contains no small cuts is equivalent to previous arguments. Hence $G'$ is a PT graph. By the minimality of $G$, $G'$ has a valid orientation. This extends to a valid orientation of $G$, a contradiction. \\

Therefore $e_3$ is incident with either two degree four vertices or with~$w$. The same must be true of the corresponding edge incident with~$v$. 

\begin{Prop}{PT9}
Both $u$ and $v$ are adjacent to $w$. 
\end{Prop}
\begin{proof}
Consider the case where $e_3$ is incident with two degree four vertices. Then $e_3$ is a chord. Let $y$ be the other endpoint of~$e_3$. Then $y$ is adjacent to $w$ on the boundary of $F_G$, else $G$ has a $2$-robust at most $5$-edge-cut of Type 2, where $t$ is not in the side contained in a disk, a contradiction. This cut is shown in Figure~\ref{ptten}. If the corresponding edge incident with $v$ is a chord, the same is true. Then $G$ has a $2$-robust at most $3$-edge-cut, a contradiction. This cut can also be seen in Figure~\ref{ptten}. Thus we may assume that $v$ and $w$ are adjacent.\\

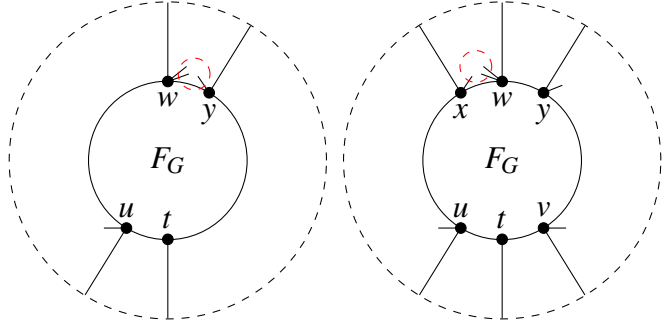
\begin{figure}
\begin{center}
\subfloat{
\begin{tikzpicture}
\draw (0,0) circle (30pt) node {$F_G$};
\filldraw (0,-1.05) circle (2pt) node[above] {$t$};
\filldraw (-0.55,-0.9) circle (2pt) node[above] {$u$};
\filldraw (0,1.05) circle (2pt) node[below] {$w$};
\filldraw (0.55,0.9) circle (2pt) node[below] {$y$};

\draw[dashed] (0,0) circle (60pt);
\draw (0,-1.05) -- (0,-2.1);
\draw (0,1.05) -- (0,2.1);
\draw (-0.55,-0.9) -- (-1.1,-1.79);
\draw (0.55,0.9) -- (1.1,1.79);
\draw (-0.55,-0.9) -- (-0.85,-0.9);

\draw (0.55,0.9) -- (0.4,1.12);
\draw (0,1.05) -- (0.25,1.25);
\draw (0,1.05) -- (0.28,1.15);

\draw[red, dashed] (0.35,1.15) circle (6pt);
\end{tikzpicture}
}
\subfloat{
\begin{tikzpicture}
\draw (0,0) circle (30pt) node {$F_G$};
\filldraw (0,-1.05) circle (2pt) node[above] {$t$};
\filldraw (0.55,-0.9) circle (2pt) node[above] {$v$};
\filldraw (-0.55,-0.9) circle (2pt) node[above] {$u$};
\filldraw (0,1.05) circle (2pt) node[below] {$w$};
\filldraw (0.55,0.9) circle (2pt) node[below] {$y$};
\filldraw (-0.55,0.9) circle (2pt) node[below] {$x$};

\draw[dashed] (0,0) circle (60pt);
\draw (0,-1.05) -- (0,-2.1);
\draw (0,1.05) -- (0,2.1);
\draw (-0.55,-0.9) -- (-1.1,-1.79);
\draw (0.55,0.9) -- (1.1,1.79);
\draw (0.55,-0.9) -- (1.1,-1.79);
\draw (-0.55,0.9) -- (-1.1,1.79);
\draw (-0.55,-0.9) -- (-0.85,-0.9);
\draw (0.55,-0.9) -- (0.85,-0.9);

\draw (0.55,0.9) -- (0.8,1);
\draw (-0.55,0.9) -- (-0.4,1.12);
\draw (0,1.05) -- (-0.25,1.25);
\draw (0,1.05) -- (-0.28,1.15);

\draw[red, dashed] (-0.35,1.25) circle (6pt);
\end{tikzpicture}
}
\caption{Edge $tw$ is a chord: Small Type 2 cuts.}
\label{ptten}
\end{center}
\end{figure}


Let $f_1$, $f_2$, $f_w$, $f_t$ be the edges incident with $v$ in order, where $f_t=vt$, $f_w=vw$, and $f_1$ is a boundary edge of~$F_G$. Now $f_2$ is not incident with two vertices of degree $4$, else $G$ has either a $2$-robust at most $4$-edge-cut, or parallel edges. Also, $v$ cannot be adjacent to $w$ via parallel edges. Orient and delete $f_1$ and $f_2$ to satisfy $p(v)$, and contract the set of vertices $\{t,w,v\}$ to a single vertex of degree $4$, calling the resulting graph~$G'$. Now $G'$ has only one possible degree three vertex, in $G$ it is incident with~$f_1$. If $G'$ contains a $2$-robust edge-cut of size at most $3$, then $G$ contains a $2$-robust edge-cut of size at most $5$ containing a boundary edge incident with $t$, a contradiction. Thus $G'$ is a PT graph, and has a valid orientation by the minimality of~$G$. Transfer this orientation to~$G$. Orient the remaining edges incident with~$t$. Since $f_1$ and $f_2$ satisfy $p(v)$, $f_w$ is known to have the opposite direction (relative to $v$) from~$f_t$. Since $w$ cannot be the only vertex whose prescription is not met, this is a valid orientation of~$G$. 
\end{proof}

The remaining case is that both $u$ and $v$ are adjacent to~$w$. Let $P_1=tu_1u_2...u_iw$ be the path on the boundary of $F_G$ from $t$ to $w$ that includes $u$, and let $P_2=tv_1v_2...v_jw$ be the path on the boundary of $F_G$ from $t$ to $w$ that includes~$v$. Let $S_1$ and $S_2$ be the subsets of vertices in $P_1$ and $P_2$ respectively that have an adjacent vertex of degree at least $5$ that is not~$w$. \\

If $S_1 \neq\emptyset$, then let $k$ be such that $u_k \in S_1$ and $\min \{k, i - k + 1.5\}$ is minimised.  If $S_2 \neq\emptyset$, then let $\ell$ be such that $v_\ell \in S_2$ and $\min \{\ell, j - \ell + 1.5\}$ is minimised. Without loss of generality, suppose that $\min \{k, i - k + 1.5\}\leq \min \{\ell, j - \ell + 1.5\}$. Let $e_1,e_2,e_3,e_4$ be the edges incident with $u_k$ in order, where if $k> \frac{i+1}{2}$, $e_1=u_{k-1}u_k$ and $e_4=u_ku_{k+1}$ (with the convention that $t=u_0$ and $w=u_{i+1}$ if necessary), and otherwise, $e_4=u_{k-1}u_k$ and $e_1=u_ku_{k+1}$. At least one of $e_1$ and $e_2$ is incident with a degree $5$ vertex, that is not $w$, by definition. \\

If, for example, $k= 3$, then $u_1 = u$ is adjacent to $w$ and $v_j$, $v_j$ is adjacent to $u_2$, $u_2$ is adjacent to $v_{j-1}$ and $v_{j-1}$ is adjacent to~$u_3$. Likewise $v_1=v$ is adjacent to $w$ and $u_i$, $u_i$ to $v_2$, $v_2$ to $u_{i-1}$, and $u_{i-1}$ to~$v_3$. More generally, the set $X$ defined next, consists of those vertices whose adjacencies are determined in this fashion.\\

If $k>\frac{i+1}{2}$, let $X=\{u_m:m\geq k\text{ or }m\leq i-k+1\}\cup\{v_m:m\geq j-i+k\text{ or }m\leq i-k+1\}$. Otherwise, let $X=\{u_m:m\leq k\text{ or }m> i-k+1\}\cup\{v_m:m\leq k\text{ or }m> j-k+1\}$. Orient and delete $e_1$ and $e_2$ to satisfy $p(u_k)$, and contract the set of vertices $X\cup\{t,w\}$ to a single vertex of degree~$4$. Call the resulting graph~$G'$. By definition, $G'$ has at most one degree three vertex. If $G'$ has a $2$-robust at most $3$-edge-cut, then $G$ has a corresponding $2$-robust at most $5$-edge-cut that does not separate $t$ from $w$, a contradiction. Hence $G'$ is a PT graph. By the minimality of $G$, $G'$ has a valid orientation. Transfer this orientation to~$G$.\\

Suppose that $k>\frac{i+1}{2}$. Orient and delete the two remaining unoriented edges incident with the following vertices in order:
$$v_{j-i+k},u_{i-k+1},v_{j-i+k+1},u_{i-k},...,v_j,u_1,t,w,v_1,u_i,v_2,u_{i-1},...,v_{i-k},u_{k+1}.$$
There is only one unoriented edge at $u_k$ ($u_kv_{i-k+1}$), which by construction must have the opposite direction (relative to $u_k$) from $u_ku_{k+1}$. Since $v_{i-k+1}$ cannot be the only vertex whose prescription is not met, this is a valid orientation for~$G$. \\

Suppose that $k\leq\frac{i+1}{2}$. Orient and delete the two remaining unoriented edges incident with the following vertices in order:
$$v_k,u_{i-k+2},v_{k-1},u_{i-k+3},...,u_i,v_1,t,w,w_1,v_j,u_2,v_{j-1},...,v_{j-k+3},u_{k-1}.$$
There is only one unoriented edge at $u_k$ ($u_kv_{j-k+2}$), which by construction must have the opposite direction (relative to $u_k$) from $u_ku_{k+1}$. Since $v_{j-k+2}$ cannot be the only vertex whose prescription is not met, this is a valid orientation for~$G$. \\

Now suppose that $S_1,S_2=\emptyset$. Then every vertex in $G$ is on the boundary of $F_G$, and aside from $t$ and $w$, all vertices have degree~$4$. Note that $P_1$ and $P_2$ have the same length, so $j=i$, and the edges in the graph that are not on the boundary of $F_G$ are
$$\{tw, u_1w, v_1w, u_kv_{i-k+1}\text{ where $1\leq k\leq i$},u_{\ell}v_{i-\ell}\text{ where $2\leq \ell\leq i$}\}$$
(See \ref{ptgraph}). Then $G=A_{i+1}$, where $i\geq 3$ is odd. Lemma \ref{circul} yields a valid orientation for $G$. 
\end{proof}
%
%

\section{Discussion}

\label{projdis}

The following result (the Strong $3$-Flow Conjecture for projective planar graphs) is a direct corollary of Theorem \ref{pmain}.

\begin{theorem}
\label{main1}
Let $G$ be a $5$-edge-connected graph embedded in the projective plane. Then $G$ is $\mathbb{Z}_3$-connected.
\end{theorem}

The case where the prescription function $p$ is such that $p(v)=0$ for all $v\in V(G)$, is the 3-Flow Conjecture for projective planar graphs, shown by \citet{steinbergyounger89}. \\

Unlike in the plane, in the projective plane (Theorem \ref{pmain}) we do not allow a directed vertex. In the plane this is necessary, in order to reduce small edge-cuts. In the projective plane we utilise the property that one of the two graphs resulting from contracting the sides of an edge-cut is a planar graph, and thus are able to apply planar results. \\

In fact, adding a directed vertex $d$ of degree $4$ to the result is not possible. Consider a graph $G$ with specified face $F_G$ such that every vertex is on the boundary of~$F_G$. Let $t$ be a degree $3$ vertex on the boundary of $F_G$, and $w$ a degree $5$ vertex on the boundary of $F_G$, adjacent to $t$ via a non-contractible chord. Let $k\in \mathbb{Z}^{\geq 0}$. Let the two paths between $t$ and $w$ on the boundary of $F_G$ be $P_1=tu_1u_2...u_nw$ and $P_2=tv_1v_2...v_nw$ where $n=3k+5$ and all vertices in $G$ aside from $t$ and $w$ have degree~$4$. We assume that for all $1\leq i\leq n$, $u_i$ is adjacent to $u_{i-1}$, $u_{i+1}$, $v_{n-i+1}$, and $v_{n-i+2}$, where we set $u_0=v_0=t$ and $u_{n+1}=v_{n+1}=w$. Finally, $w$ and $v_1$ are adjacent. We let $d$ be~$v_{n-1}$. See Figure~\ref{pstar}.\\

\begin{figure}[t]
\begin{center}
\begin{tikzpicture}[scale=1.4]
\draw (0,0) circle (100pt);

\filldraw (0,3.5) circle (2pt) node[above] {$w,0$};
\filldraw (1.08,3.33) circle (2pt) node[above] {$v_n,+1$};
\filldraw (-1.08,3.33) circle (2pt) node[above] {$u_n,+1$};
\filldraw (2.06,2.83) circle (2pt) node[right] {$d,-1$};
\filldraw (2.83,2.06) circle (2pt) node[right] {$v_{n-2},+1$};

\filldraw (0,-3.5) circle (2pt) node[below] {$t,0$};
\filldraw (1.08,-3.33) circle (2pt) node[below] {$v_1,+1$};
\filldraw (-1.08,-3.33) circle (2pt) node[below] {$u_1,-1$};
\filldraw (2.06,-2.83) circle (2pt) node[right] {$v_2,-1$};
\filldraw (-2.06,-2.83) circle (2pt) node[left] {$u_2,+1$};
\filldraw (-2.83,-2.06) circle (2pt) node[left] {$u_3,+1$};
\filldraw (-3.33,-1.08) circle (2pt) node[left] {$u_4,+1$};

\draw (0,3.5) -- (0,-3.5);
\draw (0,3.5) -- (1.08,-3.33);
\draw (1.08,-3.33) -- (-1.08,3.33);
\draw (-1.08,3.33) -- (2.06,-2.83);
\draw (2.06,-2.83) -- (1.65,-2.26);

\draw (0,3.5) -- (-1.08,-3.33);
\draw (-1.08,-3.33) -- (1.08,3.33);
\draw (1.08,3.33) -- (-2.06,-2.83);
\draw (-2.06,-2.83) -- (2.06,2.83);
\draw (2.06,2.83) -- (-2.83,-2.06);
\draw (-2.83,-2.06) -- (2.83,2.06);
\draw (2.83,2.06) -- (-3.33,-1.08);
\draw (-3.33,-1.08) -- (-2.66,-0.86);

\draw (2.49,2.49) circle (0pt) node[rotate=-45] {$>$};
\draw (1.6,3.13) circle (0pt) node[rotate=155] {$>$};
\draw (1.3,2.07) circle (0pt) node[rotate=-135] {$>$};
\draw (1.43,1.97) circle (0pt) node[rotate=-125] {$>$};
\end{tikzpicture}
\caption{A projective planar graph drawn in the plane with a directed vertex of degree $4$ and no valid orientation.}
\label{pstar}
\end{center}
\end{figure}
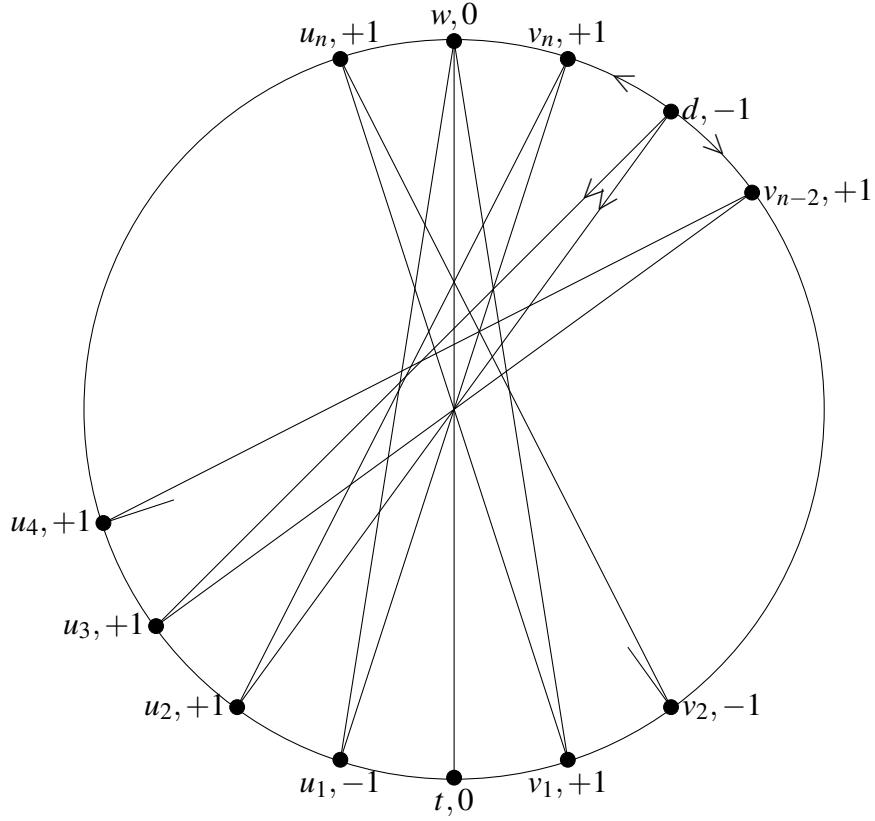


Suppose that $p(d)=-1$ and all edges incident with $d$ are directed out from~$d$. Let $p(v_n)=p(v_{n-2})=p(u_2)=p(u_3)=+1$. Let $p(u_1)=-1$. Let $p(t)=p(w)=0$. Finally, let all other vertices have prescription~$+1$. We have 
$$p(G)=4(+1)+2(-1)+2(0)+(2n+2-8)(+1)=2n-4=6k+6\equiv 0\pmod{3}.$$
Therefore $p$ is a valid prescription function for~$G$. \\

\begin{lemma}
The graph $G$ does not have a valid orientation that meets $p$.
\end{lemma}
\begin{proof}
The edges incident with $d$ meet the prescription of $v_n$, $v_{n-2}$, $u_2$, and $u_3$, hence each of these vertices has all remaining edges directed in or all remaining edges directed out. Since $v_n$ and $u_2$ are adjacent, one has all edges directed in and the other has all edges directed out. Consider~$u_1$. The edges $u_1u_2$ and $u_1v_n$ are directed one in and one out of~$u_1$. Thus the remaining edges incident with $u_1$: $u_1t$ and $u_1w$, are directed into $u_1$ to satisfy~$p(u_1)$. \\

We show that for all $1\leq j\leq \frac{n-2}{3}$, $u_{3j+1}u_{3j+2}$ and $u_{3j+1}u_{n-3j}$ are directed out of $u_{3j+1}$, and that for all $0\leq i\leq \frac{n-2}{3}$, $v_{n-(3i+1)}v_{n-(3i+1)-1}$ and $v_{n-(3i+1)}u_{3i+3}$ are directed out of $v_{n-(3i+1)}$. We consider the sequence $d, u_{4}, v_{n-4}, u_{7}, v_{n-7}, ..., u_{n-1},v_1$. The property is known to hold for~$d$. Let $x$ be the first vertex in this sequence for which this property does not hold. \\

Suppose that $x=u_{3j+1}$ for some $1\leq j\leq \frac{n-2}{3}$. Then by definition, the property holds for~$v_{n-(3j-2)}$. Hence $v_{n-(3j-2)}v_{n-(3j-2)-1}$ and $v_{n-(3j-2)}u_{3j}$ are directed out of~$v_{n-(3j-2)}$. This satisfies the prescription of $v_{n-(3j-2)-1}$ and $u_{3j}$, which are adjacent, and so the remaining three edges at one of these vertices are directed in, and the remaining three edges at the other are directed out. Both are adjacent to $u_{3j+1}$, so the edges $u_{3j+1}v_{n-(3j-2)-1}$ and $u_{3j+1}u_{3j}$ are directed one in and one out of~$u_{3j+1}$. Hence the remaining two edges incident with $u_{3j+1}$ are directed out of $u_{3j+1}$, as required. \\

Now suppose that $x=v_{n-(3j+1)}$ for some $1\leq j\leq \frac{n-2}{3}$. Then by definition, the property holds for~$u_{3j+1}$. Hence $u_{3j+1}u_{3j+2}$ and $u_{3j+1}v_{n-3j}$ are directed out of~$u_{3j+1}$. This satisfies the prescription of $u_{3j+2}$ and $v_{n-3j}$, which are adjacent, and so the remaining three edges at one of these vertices are directed in, and the remaining three edges at the other are directed out. Both are adjacent to $v_{n-(3j+1)}$, so the edges $v_{n-(3j+1)}u_{3j+2}$ and $v_{n-(3j+1)}v_{n-3j}$ are directed one in and one out of~$v_{n-(3j+1)}$. Hence the remaining two edges incident with $v_{n-(3j+1)}$ are directed out of $v_{n-(3j+1)}$, as required. \\

Therefore, $u_1t$ is directed out of $t$, and $v_1t$ is directed into~$t$. Hence no direction of $tw$ meets~$p(t)$. Thus $G$ has no valid orientation that meets~$p$.
\end{proof}

In \citet{paper1}, Conjecture 4 generalises Theorem \ref{mainft} by allowing both $t$ and $d$, provided that $deg(d)=3$. We make the analogous conjecture here, that we may have both $t$ and $d$ in Theorem \ref{pmain}, provided $deg(d)=3$.

\bibliography{ResearchProposal}
\bibliographystyle{myapalike}

\end{document}